\documentclass{article}

\usepackage[numbers]{natbib}
\usepackage{fullpage}
\usepackage{hyperref}

\usepackage{amsmath,amssymb,amsfonts, dsfont}
\usepackage{bm}
\usepackage{natbib}
\usepackage{soul,color}
\usepackage{graphicx}
\usepackage{enumitem}
\usepackage{mathtools}
\usepackage{subcaption}
\usepackage{amsthm}

\usepackage{makecell} 
\usepackage{subcaption} 
\usepackage{dblfloatfix}
\usepackage{multicol}

\newtheorem{thm}{Theorem}
\newtheorem{defn}{Definition}

\newtheorem{exmp}{Example}

\newtheorem{Problem}{Problem}

\begin{document}

\title{Stability Analysis of Switched Linear Systems with Neural Lyapunov Functions}

\author{Virginie Debauche \and Alec Edwards \and Raphaël M. Jungers \and Alessandro Abate}

\maketitle

\begin{abstract}
Neural-based, data-driven analysis and control of dynamical systems have been recently investigated and have shown great promise, e.g. for safety verification or stability analysis. Indeed, not only do neural networks allow for an entirely model-free, data-driven approach, but also for handling arbitrary complex functions via their power of representation (as opposed to, e.g. algebraic optimization techniques that are restricted to polynomial functions). Whilst classical Lyapunov techniques allow to provide a formal and robust guarantee of stability of a \emph{switched} dynamical system, very little is yet known about correctness guarantees for \emph{Neural} Lyapunov functions, nor about their performance (amount of data needed for a certain accuracy). \\
We thus formally introduce neural Lyapunov functions for the stability analysis of switched linear systems: we benchmark them on this paradigmatic problem, which is notoriously difficult (and in general Turing-undecidable), but which admits recently-developed technologies and theoretical results. Inspired by switched systems theory, we provide theoretical guarantees on the representative power of neural networks, leveraging recent results from the ML community. We additionally experimentally display how neural Lyapunov functions compete with state-of-the-art results and techniques, while admitting a wide range of improvement, both in theory and in practice. This study intends to improve our understanding of the opportunities and current limitations of neural-based data-driven analysis and control of complex dynamical systems.     
\end{abstract}

\section{Introduction}

The interaction between Lyapunov theory and learning has motivated emerging research in recent years. For instance, \cite{berkenkamp2017safe} or \cite{fazlyab2021introduction} employ Lyapunov approaches to certify safety conditions in AI-based systems. Other works (among which the present one), motivated by the great performance of neural networks in many computational technologies, tackle the other direction and use AI techniques in order to learn Lyapunov functions \cite{Chang2020, Abate2021, Dawson2021, Farsi2022, Zhang2023}. They provide promising proofs of concept for automated, data-driven control solutions that are agnostic of the particular properties of the considered control system. Indeed, a prime advantage of neural networks is that they can be deployed in an unknown setting and on complex systems, such as Cyber-Physical Systems, where classical control techniques would require verifying technical properties (such as Lipschitz continuity, linearity, or passivity) that are out of hand in practice.

A main paradigm along these lines is a situation where a neural network is constructed from observed trajectories, and its input-ouput relation is interpreted as a Lyapunov function for the system. Our work falls within this paradigm. However, a crucial limitation of neural techniques is that they rarely come with guarantees; that is, the obtained network is tailored to satisfy the properties of a Lyapunov function \emph{on the observed trajectories}, but very little is known about its generalisation properties out of the sample set. In safety-critical application, one needs generalisation guarantees, where the observed performance of the Lyapunov function is guaranteed to hold for the whole universe, not only the sampled set. Even more, Lyapunov functions must have specific properties, which might be violated by the obtained neural network (think simply of positive definiteness).

In this work, we are primarily interested in understanding this (potential) gap between high performance on the sampled set, and behaviour on the true system. For this purpose, our strategy is to study this approach on a family of complex, but still well-understood systems, for which alternative computational techniques have been developed, so that one can compare the performance of neural networks, both in terms of computational performance and in terms of accuracy. In this paper we focus on switched systems. 
 
Switched systems have attracted massive research efforts in Systems and Control, because they provide a relatively simple framework for representing many complex engineering applications. As a few examples, they have been used in bipedal robotics \cite{Brogliato2004}, in image processing \cite[Chapter 5]{RJ-2009}, or multihop control networks \cite{Alur2011}. Closer to AI, switched systems have been used to model Q-learning algorithms \cite{Lee2020}, and techniques developed for classification have been used in switched system identification \cite{Lauer2008}.

In particular, we focus on the stability analysis problem for switched systems, which is encapsulated by the notion of \emph{joint spectral radius} (in short, JSR). It has been proven that it is extremely hard to compute the JSR \cite{BLONDEL2000}, and in full generality, the problem of computing this quantity is Turing-undecidable \cite{Blondel_boundedness_2000}. However, different computational techniques have been developed, relying essentially on Lyapunov theory, with good practical performance. One of the main families of such techniques rely on the power of Semi-Definite Programming (SDP), and provides quadratic/ellipsoidal \cite{BlonNesThe-2005} or Sum-Of-Squares (SOS) \cite{PARRILO2008} Lyapunov functions. Despite the hardness of the JSR computation, these references provide approximation guarantees; that is, one can obtain an estimate of the JSR, with a bound on the error made by the estimate. See \cite{RJ-2009} for a whole review on this topic. 

As a first theoretical contribution, we provide similar theoretical approximation guarantees for neural network-based Lyapunov functions. For this, we leverage results from convex geometry, and combine them with recent results in Machine Learning (ML), about the representation power of neural networks. Results on the approximation power of neural networks are well established, but have not previously been tailored to Lyapunov functions constructed as neural networks.
 
Our second contribution is empirical. It is well known that ReLU neural networks compute piecewise linear functions. Such functions are another popular technique to compute the JSR, however with less efficient computational power: the computation of a piecewise linear norm cannot be achieved efficiently since it requires to solve a bilinear program \cite{AthaJung-2019}, which cannot be done in polynomial time. ReLU neural networks offer an alternative to compute this type of Lyapunov functions. We thus provide numerical experiments in order to benchmark this alternative way of obtaining piecewise linear Lyapunov functions. We first consider a low-dimensional switched system as proof of concept where we observe that we are competitive in terms of approximation precision. Then we increase the dimension to show that we compete with SDP-based techniques both in terms of computation time and approximation precision.
 
Towards the end of the paper, we provide an overview of technical problems that are encountered, and investigate techniques from ML and from control to alleviate these problems.

\subsection{Notation}
Given an integer $m \in \mathbb{N}$, the notation $\langle m \rangle$ refers to all the integers smaller than or equal to $m$, i.e. $ \left\{ 1, \dots, m \right\}$. Given a real number $x \in \mathbb{R}$, $\lfloor x \rfloor$ denotes the floor of $x$, i.e. the largest integer which does not exceed $x$, while $\lceil x \rceil$ denotes the ceiling of $x$, i.e. the smallest integer exceeding $x$.

\section{Background}
\subsection{Switched systems and joint spectral radius}

In classical linear systems theory, a typical question is whether the \emph{discrete-time linear system} $ x_{k+1} := A \, x_k $ is \emph{stable}, that is, whether all the trajectories asymptotically tend to zero. A necessary and sufficient condition for this is that the \emph{spectral radius} of the square matrix $A \in \mathbb{R}^{n \times n}$ with $n \in \mathbb{N}$ defined by 
\[ \rho(A) ~ := ~ \lim\limits_{k \to \infty} \left\| A^k \right\|^{1/k}\]
\noindent is less than one. The spectral radius expresses the maximal rate of growth of the system, and has been proved to be equal to the maximal modulus of the eigenvalues of the matrix $A$. 

Discrete-time linear systems can be generalised to \emph{switched linear systems} defined by
\begin{equation}\label{Eq:LinSwitchedSystem}
    x_{k+1} ~ = ~ A_{\sigma(k)} \, x_k
\end{equation}
where the dynamics evolves at each time step according to a finite set of square matrices $\Sigma := \left\{ A_1, \dots, A_M \right\} \subset \mathbb{R}^{n \times n}$, $n \in \mathbb{N}$, and the \emph{switching signal} $\sigma : \mathbb{N} \to \langle M \rangle$. Here, the switching signal can be interpreted as an exogenous perturbation (think of an operator who may switch the operating mode of a system, or any other situation where the law of dynamics may switch from time to time e.g. due to external disturbance, or change of specification) and the system is thus defined by \emph{several} matrices rather than a single one, as is the case for classical linear systems. A switched system is \emph{(worst-case) stable} if for all $x_0\in \mathbb R^n,$ for all switching signals $\sigma(.),$ $x_k \rightarrow 0$ when $k\rightarrow \infty.$

For switched systems, one can generalise the notion of \emph{spectral radius}~\cite{Rota1960} to the finite set of square matrices $\Sigma$.

\begin{defn}[Joint spectral radius]
Given a finite set of square matrices $\Sigma := \left\{ A_1, \dots, A_M\right\} \subset \mathbb{R}^{n \times n}$ with $n \in \mathbb{N}$, the \emph{joint spectral radius} (JSR) of $\Sigma$, denoted by $\rho(\Sigma)$, is defined by
\begin{equation}\label{Eq:DefJSR}
     \rho (\Sigma) ~:=~ \limsup_{k \to \infty} \left\{ \left\| A \right\|^{1/k} : A \in \Sigma^k \right\}
\end{equation}
where $\Sigma^k := \left\{ A_{\sigma_k} := A_{\sigma_k(1)} \cdots A_{\sigma_k(k)} \mid  \sigma_k : \langle k \rangle \to \langle M \rangle \right\}$ encodes all the possible products of length $k$ of the matrices in $\Sigma$.
\end{defn} 
\noindent Similar to the linear case above, the JSR actually expresses the maximal asymptotic behaviour of a switched linear system. 
The joint spectral radius of $\Sigma$ can be related to the stability of \eqref{Eq:LinSwitchedSystem} by the following Theorem.
\begin{thm}[Corollary~1.1 in \cite{RJ-2009}]
Given a finite set of $M$ square matrices $\Sigma := \{ A_1, \dots, A_M \}$. The corresponding switched linear system is stable if and only if $\rho(\Sigma) < 1$.
\end{thm}

\subsection{Approximation of the JSR}

Unfortunately, the JSR does not share the same nice algebraic properties as the spectral radius and is extremely hard to compute in practice. Indeed, approximating the JSR is NP-hard \cite{BLONDEL1997}, whereas deciding whether the JSR is smaller than one or not is Turing-indecidable \cite{Blondel2008}, and there does not exist any algebraic criterion to decide the stability (non-algebraicity) of switched systems \cite{Kozyakin1990}. Despite these theoretical limitations, the approximation of the JSR has been tackled by many researchers. In particular, Lyapunov methods have been exploited thanks to the following characterization of the JSR.
\begin{thm}[Proposition~1 in \cite{Rota1960}] \label{Thm:CharacJSR}
For any bounded set of matrices $\Sigma$ such that $\rho(\Sigma) \neq 0$, the joint spectral radius of $\Sigma$ can be defined as 
\begin{equation} \label{eq-infnorm} 
\rho(\Sigma) ~:= ~ \inf_{\left\| \cdot \right\|} \, \max_{A \in \Sigma} \, \left\| A \right\|. 
\end{equation}
\end{thm}
\noindent This result in particular tells that any norm $\| \cdot \| : \mathbb{R}^n \to \mathbb{R}_{\geq 0}$ provides an upper-approximation of the JSR. Therefore, one can in principle compute the JSR of a finite set of matrices $\Sigma$ by looking at all norms and their corresponding approximation $\rho(\Sigma, \| \cdot \|):=\max_{A \in \Sigma} \, \left\| A \right\|$, and keep the smallest. 

In practice, one usually looks through a subset of norms for which the computation can be more easily performed. This is for instance the case for the ellipsoidal approximation of the JSR, defined as 
\begin{equation}\label{Eq:EllipsoidalApproxJSR}
\rho_{\mathcal{Q}}(\Sigma) ~:=~ \inf_{P \succ 0} \, \max_{A \in \Sigma} \, \left\| A \right\|_P 
\end{equation}
where $\left\| A \right\|_P$ denotes the matrix norm induced by the ellipsoidal vector norm associated to $P$, i.e. $\| x \|_P := \sqrt{x^\top P x}$. Note that this denomination stems from the shape of the unit ball $\mathcal{B}_{\parallel \cdot \parallel_P}$ of this norm. From a computational point of view, the ellipsoidal approximation $\rho_{\mathcal{Q}}$ can be computed efficiently thanks to SDP, see \cite{BlonNesThe-2005} for details. Moreover, the following theorem provides theoretical guarantees on the accuracy of the technique. 
\begin{thm}[Theorem~14 in \cite{BlonNesThe-2005}] \label{Thm:Quad_bound}
    Let $\rho(\Sigma)$ be the joint spectral radius of a finite set of matrices $\Sigma$ of dimension $n \in \mathbb{N}$. Then 
    \begin{equation} \label{Eq:ApproxGuaranteesQuad}
    \frac{1}{\tau_{\mathcal{Q}}} \rho_{\mathcal{Q}}(\Sigma) ~\leq ~\rho(\Sigma) ~ \leq ~ \rho_{\mathcal{Q}}(\Sigma),
    \end{equation}
    where $\tau_{\mathcal{Q}} := \sqrt{n}$.
\end{thm}
\noindent A similar SDP-based approach using SOS polynomials instead of quadratic functions has been developed \cite[Theorem~3.4]{PARRILO2008}. In this case, the SOS approximation of degree $2d$, $d \in \mathbb{N}$ of the JSR denoted by $\rho_{SOS, 2d}(\Sigma)$ satisfies the following inequalities 
\begin{equation} \label{Eq:ApproxGuaranteesSOS}
\frac{1}{\tau_{SOS}(n,d)^{1/2d}} \, \rho_{SOS,2d} (\Sigma) ~ \leq ~ \rho(\Sigma) ~ \leq ~ \rho_{SOS,2d} (\Sigma), 
\end{equation}
\noindent where $\tau_{SOS}(n,d) = \begin{pmatrix} n+d-1 \\ d \end{pmatrix}$. See \cite{RJ-2009} for a comprehensive review on the JSR.

\subsection{ReLU-activated Neural Nets, CPWL functions}

We consider neural networks with ReLU activation functions. 

\begin{defn}[ReLU neural network]
A \emph{Rectified Linear Units (ReLU) feedforward neural network} with $k +1 \in \mathbb{N}$ hidden layers is defined by $k$ affine transformations $T^{(j)} : \mathbb{R}^{n_{j-1}} \to \mathbb{R}^{n_j}$, $ x \mapsto W^{(j)} x + b^{(j)}$ for $j \in \langle k \rangle$, and a linear transformation $T^{(k+1)} : \mathbb{R}^{n_{k}} \to \mathbb{R}^{n_{k+1}}$ , $ x \mapsto W^{(k+1)} x$. The network \emph{represents} the function $V : \mathbb{R}^{n_0} \to \mathbb{R}^{n_{k+1}}$ given by 
\begin{equation} \label{eq-neuralfunction} 
V ~:= ~ T^{(k+1)} \circ \sigma \circ \cdots \circ T^{(2)} \circ \sigma \circ T^{(1)},
\end{equation}
where $\sigma$ is the component-wise rectifier function, i.e. $\sigma(x) = \left( \max \{0, x_1 \}, \dots, \max \{0, x_n\} \right)$. The matrices $W^{(l)}$ and the vectors $b^{(l)}$ are respectively the \emph{weights} and the \emph{biases} of the $l$-th layer while $n_l$ is the \emph{width} of the $l$-th layer. The maximum width of all the hidden layers is called the \emph{width} of the neural network, and the \emph{depth} is $k+1$. 
\end{defn}
\noindent Any function represented by a ReLU neural network is a continuous piecewise affine (CPWA) function whose sublevel sets are polytopes.
\begin{defn}[Continuous piecewise affine/linear function]
    We say that a function $f: \mathbb{R}^n \to \mathbb{R}$ is a \emph{continuous piecewise affine (resp. linear) function} (CPWA/L function) if there exists a finite set of polyhedra whose union is $\mathbb{R}^n$, and furthermore if $f$ is affine (resp. linear) over each polyhedron. The \emph{number of pieces} of $f$ is the minimal number of polyhedra necessary to express f as above.
\end{defn}
\noindent This function is positively homogeneous if and only if all the biases are zero, see \cite[Proposition~2.3]{HeretAl-2021}.

\subsection{Expressivity Power of ReLU Neural Networks}

It is well-known \cite{HORNIK1991} that ReLU neural networks can approximate arbitrarily well any continuous function, even with a single hidden layer. More recently, researchers \cite{Aroraetal-2018, HeretAl-2021, Dereich2021} have provided quantified approximation guarantees in terms of complexity of the represented function, as a function of the depth and width of the network. In the theorem below, the complexity of the function is expressed by the number of its \emph{pieces}.

\begin{thm}[Theorem~4.4 in \cite{HeretAl-2021}]\label{Thm:BoundStructPieces}
    Let $f: \mathbb{R}^n \to \mathbb{R}$ be a CPWL function with $p$ affine pieces. Then $f$ can be represented by a ReLU neural network with depth $\lceil \log_2(n+1) \rceil +1$ and width $\mathcal{O}(p^{2n^2+3n+1})$.
\end{thm}
\noindent Note that the depth of the network depends exclusively on the input size, and not on the number of pieces.

\section{Problem Statement}

As a paradigm of safety constraint, we now study the ability of neural networks to provide guarantees for the stability of a given switched linear system. In view of Theorem~\ref{Thm:BoundStructPieces} above, we are interested in approximating the JSR of a finite set of matrices using continuous piecewise linear (CPWL) functions, which we call a \emph{CPWL approximation of the JSR.}

\begin{Problem} Can one provide a guarantee on the accuracy of the best neural approximation of the JSR, that is, on the bound obtained from (\ref{eq-infnorm}) when the norms $\|\cdot{} \|$ are restricted to be of the form (\ref{eq-neuralfunction})? What is the relation between such a guarantee and the number of layers, of neurons, of the width of the network? How can one guarantee that the output $V$ of the network is indeed a valid norm?
\end{Problem}


\section{Theoretical Results}

\subsection{CPWL Approximation Guarantees} 

Inspired by the approximation results in~\eqref{Eq:ApproxGuaranteesQuad} and~\eqref{Eq:ApproxGuaranteesSOS} we exploit Theorem~\ref{Thm:CharacJSR}, additionally leveraging the following result from Barvinok about the approximation power of a convex set by a polytope as a function of the dimension and the number of vertices in the approximation polytope. 

\begin{thm}[Theorem~1.1 in \cite{Barvinok-2012}]\label{Thm:Barvinok_poly}
    Let $\tau > 1$ be a real number and $n$ and $k$ be positive integers and such that 
    \begin{equation}\label{Eq:Thm_Barvinok_poly}
         \left( \tau - \sqrt{\tau^2 -1} \right)^k + \left( \tau + \sqrt{\tau^2 -1} \right)^k ~\geq~ 6 ~ D(n,k)^{1/2}, 
    \end{equation}
    where 
    \begin{equation} \label{Eq:New_Bound_vertices_polytope} D(n,k)~:=~ \sum_{m=0}^{\lfloor k/2 \rfloor} \begin{pmatrix}
    n+k-1-2m \\ k-2m
    \end{pmatrix}.\end{equation}
    Then, for any symmetric, compact and convex set $K \subset \mathbb{R}^n$ with non-empty interior and containing the origin, there exists a symmetric polytope $P \subset \mathbb{R}^n$ with at most $8 D(n,k)$ vertices such that 
    \begin{equation}\label{Eq:IncluConvex} P ~\subseteq ~ K ~\subseteq  ~ \tau P. \end{equation}
\end{thm}
\noindent We can now derive approximation guarantees on the CPWL approximation of the JSR.

\begin{thm} \label{Thm:Poly_bound_JSR}
    Let $\rho(\Sigma)$ be the joint spectral radius of a finite set of matrices $\Sigma$ of dimension $n \in \mathbb{N}$. For any $\tau > 1$ for which there exists $k_{\tau} \in \mathbb{N}$ such that relation~\eqref{Eq:Thm_Barvinok_poly} is satisfied, the following relation holds : 
    \begin{equation} \label{Eq:PolytopicApproxGuaranttes}  \frac{1}{\tau}\, \rho_{\mathcal{P}}(\Sigma) ~ \leq \rho(\Sigma) ~ \leq ~ \rho_{\mathcal{P}}(\Sigma),
     \end{equation}
    where $\rho_{\mathcal{P}}$ is the optimal solution of (\ref{eq-infnorm}) where the norms are restricted to CPWL norms with at most $8D(n,k)$ vertices.
\end{thm}
\begin{proof}
Consider a finite set of $m$ matrices $\Sigma := \{A_1, \dots, A_m \} \subset \mathbb{R}^{n \times n}$ of dimension $n \in \mathbb{N}$, and $\rho(\Sigma)$ its joint spectral radius. By Theorem~\ref{Thm:CharacJSR} and by definition of an infimum, for any value $\varepsilon > 0$, there exists an $\varepsilon$-norm $\left\| \cdot \right\|_{\star}$ such that $\forall i = 1, \dots, m$
\[ \forall x \in \mathbb{R}^n, ~ \left\| A_i x \right\|_{\star} ~\leq~ (\rho(\Sigma) + \varepsilon) \left\| x \right\|_{\star}. \]
This norm defines a convex set $K \subset \mathbb{R}^n$ (which contains the origin) that can be approximated by a polytope. By Theorem~\ref{Thm:Barvinok_poly}, for any positive integer $d$ and any real number $\tau > 1$ satifying~\eqref{Eq:Thm_Barvinok_poly}, there exists a symmetric polytope $P \subset \mathbb{R}^n$ with at most $8 D(n,k)$ vertices such that 
\[ P ~ \subseteq ~ K ~ \subseteq ~ \tau P.\]
Therefore, the homogeneous function $V(x)$ whose 1-level set is $P$, satisfies that for all $x \in \mathbb{R}^n$, 
\[ V(x) ~ \leq ~ \left\| x \right\|_{\star} ~ \leq ~ \tau \, V(x), \] 
and then for any $i = 1, \dots, m$ and any $x \in \mathbb{R}^n$: 
\[ 
\begin{array}{rcl}
V(A_i x) & \leq & \left\| A_i x \right\|_{\star}, \\[0.2cm]
& \leq & (\rho(\Sigma) + \varepsilon) \left\| x \right\|_{\star}, \\[0.2cm]
& \leq & (\rho(\Sigma) + \varepsilon) \, \tau \, V(x).
\end{array} \]
Then, $V(A_i) \leq (\rho(\Sigma) + \varepsilon) \tau$ for all $i = 1 , \dots, m$ and for any $\varepsilon > 0$. Therefore, at the worst case, we have that 
\[ \frac{1}{\tau} \, \rho_{\mathcal{P}}(\Sigma) ~ \leq ~ \rho(\Sigma)~ \leq ~ \rho_{\mathcal{P}}(\Sigma),\] 
where the second inequality is direct since we consider a subset of norms. \end{proof}

\begin{figure}[t!]
    \centering
    \includegraphics[scale=0.33]{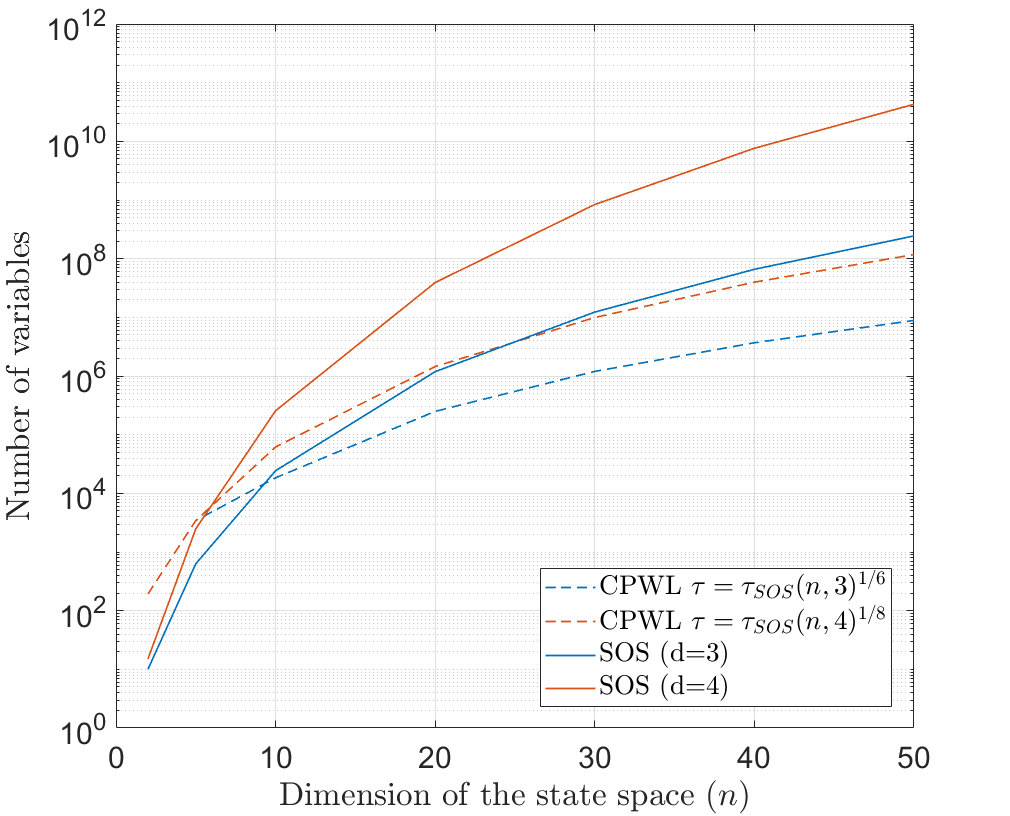}
    \caption{Evolution, as a function of the dimension ($n$), of the number of variables required to achieve precision values $\tau$ on the JSR approximation with the CPWL approach (dashed line) and with the SOS approach (continuous line). In blue, $\tau = \tau_{SOS}(n,3)$, in red, $\tau = \tau_{SOS}(n,4)$. These values are the minimal required accuracy to outperform the  guarantee~\eqref{Eq:ApproxGuaranteesSOS} for, respectively, degree-$3$ and -$4$ SOS. One can see that for large values of $n$, the CPWL approximation needs far fewer variables, for any given precision.} 
    \label{Fig:ComparisonNbVariables_Poly_Quad_SOS}
\end{figure}

\noindent Using this theorem, we can compute an upper-bound on the number of variables required to achieve a given precision $\tau$ as the number of vertices multiplied by the dimension of the state space. In Figure~\ref{Fig:ComparisonNbVariables_Poly_Quad_SOS}, we show the evolution with the dimension of this number of variables, for different values of $\tau$. For the sake of comparing with the performance of the SOS approximations of degree $3$ and $4$, we selected $\tau = \tau_{SOS}(n,d)$ for $d = 3$ and $d=4$. These results show that for small dimensions, the CPWL approach requires slightly more variables than the SDP solution. However, when we consider higher dimensions, the trend is reversed. These observations motivate the CPWL approach rather than the polynomial approach since we need less variables in high dimension. However, computation methods rapidly suffer from the curse of dimensionality: the complexity is polynomial for SDPs while it is exponential for the CPWL norms, which necessitate solving a bilinear program as explained previously. For this reason, we now investigate generating CPWL approximations with neural networks.

\subsection{Bounds on the Structure of the Network} 

Theorem~\ref{Thm:Poly_bound_JSR} provides a relation between the CPWL approximation precision and the number of vertices of its polytopic sublevel sets. Now, combining it with Theorem~\ref{Thm:BoundStructPieces}, we are able to derive bounds on the network structure (i.e. its depth and its width) to be able to achieve this precision. 

To do so, we first need to bound the number of faces of a polytope, i.e. the number of pieces of the corresponding CPWL function, given the number of vertices. 

\begin{thm}[\cite{mcmullen_1970}] \label{Thm:McMullen}
    Let $P$ be a polytope in dimension $n \in \mathbb{N}$ with $k$ vertices. Then, $P$ has at most 
    \[ \begin{pmatrix} k - \lfloor \frac{n+1}{2} \rfloor \\ k-n
        \end{pmatrix} + \begin{pmatrix} k - \lfloor \frac{n+2}{2} \rfloor \\ k-n
        \end{pmatrix} \]
    faces.
\end{thm}

We are finally able to state the following theorem, which provides original theoretical bounds on the structure of a ReLU neural network to provide a CPWL approximation of the JSR with a given precision $\tau$. To the best of our knowledge, this is the first universal approximation result tailored to neural Lyapunov functions.

\begin{thm}\label{Thm:BoundStructNN}
    Let $\rho(\Sigma)$ be the joint spectral radius of a finite set of matrices $\Sigma$ of dimension $n$. For any real $\tau > 1$ for which there exists $k_{\tau} \in \mathbb{N}$ satisfying relation~\eqref{Eq:Thm_Barvinok_poly}, there exists a CPWL function represented by a ReLU neural network of depth 
    \[ \lceil \log_2(n+1) \rceil +1 \]
    and width 
    \[  \mathcal{O} \left(  \left[ \begin{pmatrix} D_n^{\tau} - \lfloor \frac{n+1}{2} \rfloor \\ D_n^{\tau}-n
        \end{pmatrix} + \begin{pmatrix} D_n^{\tau} - \lfloor \frac{n+2}{2} \rfloor \\ D_n^{\tau}-n
        \end{pmatrix} \right]^{2n^2 +3n +1} \right) \]
    where $ D_n^{\tau} := 8 D(n,k_{\tau})$ as defined in~\eqref{Eq:New_Bound_vertices_polytope}, which approximates $\rho(\Sigma)$ with a precision of $\tau$.
\end{thm}
\begin{proof}
    This theorem results from the successive application of Theorems~\ref{Thm:Poly_bound_JSR},~\ref{Thm:McMullen} and~\ref{Thm:BoundStructPieces}.
\end{proof}

\begin{figure}[b!]
    \centering
    \begin{subfigure}[t]{0.32\linewidth}
        \centering
        \includegraphics[scale=0.31]{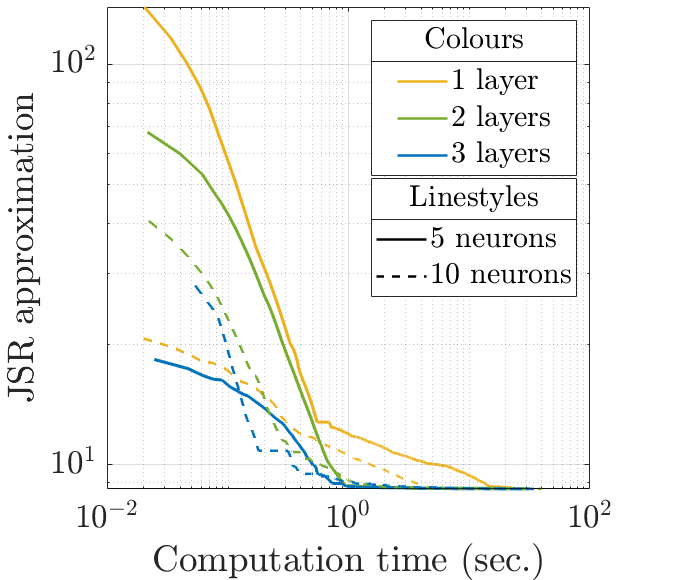}
        \caption{Evolution with the computation time of the loss function of a ReLU neural network with different widths and depths. For each configuration, only the best seed (which provides the lowest approximation) is illustrated.} 
        \label{Fig:Dimension_2_best_seeds}
    \end{subfigure}
    \hfill 
    \begin{subfigure}[t]{0.32\linewidth}
        \centering
        \includegraphics[scale=0.31]{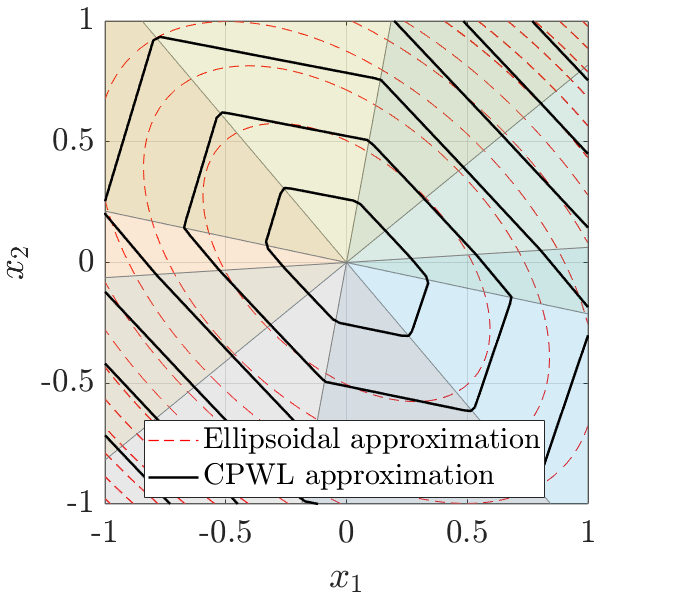}
        \caption{Sublevel sets of the best CPWL approximation with $1$ hidden layer and $5$ neurons, and the partition induced by the network. The sublevel sets of the ellipsoidal approximation have been added for comparison.} 
        \label{Fig:Dimension_2_sublevel_sets_5}
    \end{subfigure}
    \hfill
    \begin{subfigure}[t]{0.32\linewidth}
        \centering
        \includegraphics[scale=0.31]{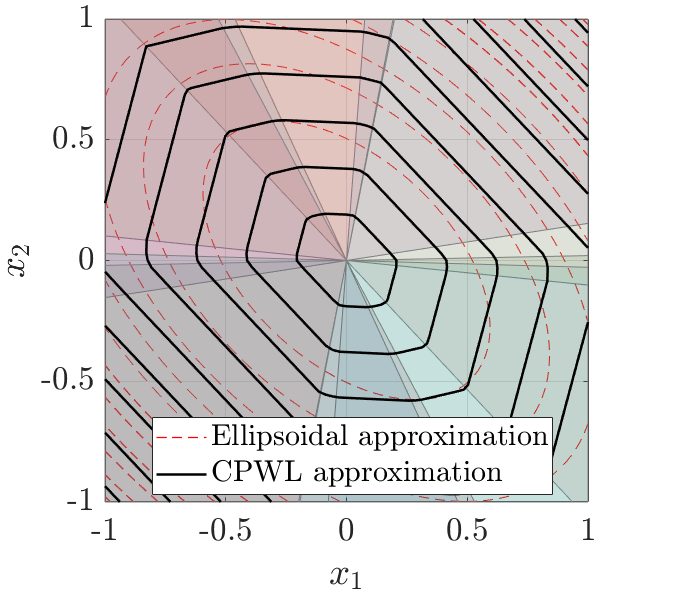}
        \caption{Sublevel sets of the best CPWL approximation with $1$ hidden layer and $10$ neurons, and the partition induced by the network. The sublevel sets of the ellipsoidal approximation have been added for comparison.} 
        \label{Fig:Dimension_2_sublevel_sets_10}
    \end{subfigure}
    \caption{Approximation of the JSR of system~\eqref{Eq:ExmpNumSS} using a ReLU neural network with $500$ sample points, alongside sublevel set of the best approximation from two networks. With 5 neurons~(\ref{Fig:Dimension_2_sublevel_sets_5}) , the CPWL approximation appears geometrically similar to the ellipsoidal approximation. As more neurons are added~(\ref{Fig:Dimension_2_sublevel_sets_10}), the geometry of the sublevel sets diverges from the ellipsoidal approximation, corresponding to an improved approximation of the JSR~(Table \ref{Table:ApproxJSR_SDP_vs_ReLUNN_dim2}). }
    \label{Fig:NumExpDimension2}
\end{figure}

\section{Experimental Evaluation} 

The above theoretical results provide a proof of concept that neural network can approximate the JSR up to an a priori fixed guarantee of accuracy. However, all the bounds above are on the worst case, and we now provide an empirical investigation of the practical efficiency of the neural approach, in comparison with the more classical SDP-based approach.

\subsection{Setup}
Given a switched linear system~\eqref{Eq:LinSwitchedSystem} in dimension $n$, we consider a ReLU neural network without biases. This is without loss of generality, since we may restrict ourself to homogeneous functions $V.$ We train the network with a sample set $\mathcal{S}$ of data points uniformly distributed on the unit sphere $\mathcal{S}^{n-1} := \{ x \in \mathbb{R}^n : \left\| x \right\|_2 = 1\}$ (by linearity and homogeneity, it is sufficient to consider sample points on the sphere). We structurally enforce during training the positivity of the represented function $V(\cdot)$ by imposing nonnegativity for the output weights. We consider the loss function defined by 
\begin{equation}
  \label{eq-neuralnorm}  L(V, \mathcal{S}) ~:=~ \max_{i = 1, \dots, M} \, \max_{x \in \mathcal{S}} \, \frac{V(A_i x)}{V(x)},
\end{equation}
which provides a sample-based approximation of (\ref{eq-infnorm}) where the norms are of the form $V(\cdot),$ as represented by the network. 
Therefore, given a ReLU neural network with $k$ hidden layers and $m$ neurons in each of them, the \emph{neural approximation} of the JSR denoted by $\rho_{NN(k,m)}(\Sigma)$ is defined as the lowest value of the loss function during the entire training campaign. Note that $\rho_{NN(k,m)}(\Sigma)$ depends on the sampled set $\mathcal{S}$ and for this reason, the quantity $\max_{x \in \mathcal{S}} \, \frac{V(A_i x)}{V(x)}$ is not formally a norm and thus nothing guarantees that the network provides a formal upper bound on the JSR (it would be if $\mathcal{S}$ was replaced by $\mathbb R^n,$ in Equation (\ref{eq-neuralnorm}) above). 

\subsection{Numerical Results}

As reference benchmarks, we consider two switched linear systems in dimension $2$ and $8$ respectively, which are challenging as they are known to lead to poor approximation with classical techniques (see \cite[Section~5]{DebDelRosJun-2023} and \cite[Example~5.2]{Ahmadi2014}). The first one serves as proof of concept, while the second example shows that in higher dimensions, neural Lyapunov functions are competitive with SOS-based techniques both in terms of precision and (as suggested earlier in Figure \ref{Fig:ComparisonNbVariables_Poly_Quad_SOS}) in computation time. Note that all experiments were run on an Intel i7 laptop with 4 cores and 8GB of RAM. 

\begin{exmp}[Low dimension] \label{Exmp:Dimension2}
    We consider the switched linear system $\Sigma_2 := \{ A_1, A_2\} \subset \mathbb{R}^{2 \times 2}$ defined by
    \begin{equation}\label{Eq:ExmpNumSS}
    A_1 = \left[ \begin{matrix} 1.5519 &  0.4474 \\
    7.6412  &  7.4716 \end{matrix} \right] ~\text{and}~ A_2 = \left[ \begin{matrix} 0.4750  &  9.1755 \\
    1.8955  &  0.1850 \end{matrix} \right].
\end{equation}

\begin{table}[t!]
    \centering
    \begin{tabular}{|cc|ccc|}
        \Xhline{1.2pt}  
        & &  \multicolumn{3}{c|}{\bf Neural approximation} \\
        & & \multicolumn{3}{c|}{$\rho_{NN(k,m)}(\Sigma_2)$} \\[0.2cm]
        \hline
        $k$ & $m$ & Best & Mean & Std.\\
        \Xhline{1.2pt}
        1 layer & 5 neurons & $8.6977$ & $9.0251$ & $0.8800$ \\
                & 10 neurons & $8.6910$ & $8.6969$ & $0.0056$ \\ 
        \hline
        2 layers & 5 neurons & $8.6983$ & $8.9312$ & $0.4645$\\
                 & 10 neurons & $8.6944$ & $8.7049$ & $0.0077$ \\
        \hline
        3 layers & 5 neurons & $8.6967$ & $9.1984$ & $0.7293$ \\
                 & 10 neurons & $8.6946$ & $8.7130$ & $0.0175$ \\
        \Xhline{1.2pt}
    \end{tabular}
    \caption{Best and mean/std. (over $20$ seeds) approximation of the JSR of system~\eqref{Eq:ExmpNumSS} provided by a ReLU neural networks with different architectures.}
    \label{Table:ApproxJSR_SDP_vs_ReLUNN_dim2}
\end{table}

\noindent The joint spectral radius of $\Sigma_2$ is $8.6881$, the ellipsoidal approximation $\rho_{\mathcal{Q}}(\Sigma_2)$ is $9.5868$ and the SOS approximation of degree $4$, i.e. $\rho_{SOS,4}(\Sigma_2)$, is $8.7203$. We consider a ReLU neural network with different depths and widths that we train with $500$ sample points for $20$ different seeds. The best (i.e. the smallest) approximation, the mean and the standard deviation are summarized in Table~\ref{Table:ApproxJSR_SDP_vs_ReLUNN_dim2}. 

Not surprisingly, the more neurons there are, the better the approximation. We also observe that with more neurons, there is significantly less variability with respect to the seed. In terms of approximation precision, the neural approach is more efficient than both the ellipsoidal and the SOS approaches, since the average approximation with $10$ neurons is smaller than $\rho_{\mathcal{Q}}(\Sigma_2)$ and $\rho_{SOS,4}(\Sigma_2)$. Figures~\ref{Fig:Dimension_2_sublevel_sets_5} and~\ref{Fig:Dimension_2_sublevel_sets_10} illustrate the partition and the sublevel sets of the Lyapunov function encoded by the best 1-layer network with $5$ and $10$ neurons respectively, and one can notice the similarity with the ellipsoidal sublevel sets. However, the computation of the ellipsoidal and the SOS approximation are very fast in such low dimension (they clock $0.2816$ and $1.5156$ seconds respectively), and the novel neural network approach cannot compete with them: for example, Figure~\ref{Fig:Dimension_2_best_seeds} shows that in the best case, we need a few seconds, but in the worse cases, we need almost one minute. 
\end{exmp}

\noindent Let us now consider a $8$-dimensional switched system, such that the computation time of the ellipsoidal and SOS approximation are larger and the network starts to be competitive.

\begin{exmp}[High dimension] \label{Exmp:Dimension8}
    We consider a switched system in dimension $8$ with $8$ different modes, defined by the matrices $\Sigma_{8} := \left\{ A_i \right\} \subset \{0,1\}^{8 \times 8}$ such that for $i = 2, \ldots, 8$,
    \begin{equation}\label{Eq:Benchmark_dim_n}
    A_i(k,l) ~ := ~ \left\{ \begin{array}{rl}
    -1 & \text{if } k = l = i, \\
    1 & \text{if } l = i \text{ and } k \neq i, \\
    0 & \text{otherwise},
    \end{array} \right.
    \end{equation} 
     while the matrix $A_1 = \textbf{1} e_1^\top$. One can prove that the joint spectral radius of this finite set of matrices is $1$, i.e. $\rho(\Sigma_{8}) = 1$. Regarding the approximation of the JSR using classical techniques, the ellipsoidal approximation generates $\rho_{\mathcal{Q}}(\Sigma_{8}) = 2.4286$ but the computation is relatively fast ($\approx 2.5$ seconds), while the SOS approximation is better since $\rho_{SOS, 4}(\Sigma_{8}) = 1.0006$ but it requires much more computation time ($\approx 258$ seconds). 

    For the neural approximation of the JSR, we use a ReLU neural network with a single hidden layer and different numbers of neurons: $10$, $15$ and $30$. We train it using $500$ sample points. Figure~\ref{Fig:Average_graph_dim_8_500_sample_points} shows the evolution of the mean and the min-max area of the JSR approximation with the computation time for each configuration. One can see that in dimension $8$, the neural approach is almost as fast as the SOS ($d=2$) method. However, the network overfits the data and then provides an approximation which is smaller than the true JSR value, except with $10$ neurons where half of the seeds provide an approximation larger than $1$. One way to prevent this behaviour is to consider more sample points. However, the computation time increases as the sample set grows and we are no longer competitive with SDP-based techniques.
\end{exmp}

\begin{figure}[t!]
    \centering
    \vspace{-0.2cm}
    \includegraphics[scale=0.31]{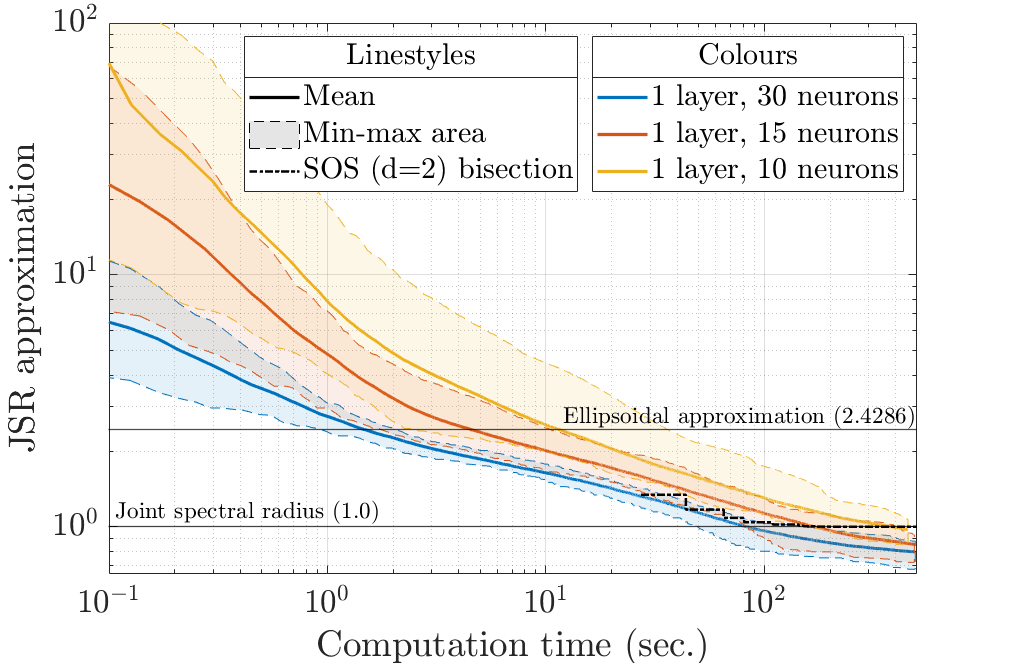}
    \caption{Evolution of the JSR approximation for the 8-dimensional system~\eqref{Eq:Benchmark_dim_n} provided by a ReLU neural network with different number of hidden layers and different numbers of neurons. The value of the JSR, the ellipsoidal approximation and the SOS degree-$4$ (in black) bisection have been added for comparison.}
\label{Fig:Average_graph_dim_8_500_sample_points}
\end{figure}

\subsection{Techniques for Improved Performance}

As we have seen, the numerical experiments suffer from classical issues in ML such as overfitting, large number of sample points needed, etc. In this section, we consider some methods to mitigate these behaviours. 

\subsubsection{Regularization}
In Figure~\ref{Fig:Average_graph_dim_8_500_sample_points}, we have seen that neural networks tend to overfit the data resulting in an incorrect estimate of the JSR. One classical method to prevent overfitting is regularization, i.e. for the problem at hand we penalize the network by adding a new term in the loss function to prevent it from learning overly-complex functions. Figure~\ref{Fig:Regularization} shows promising results for system~\eqref{Eq:Benchmark_dim_n} using L$1$-regularization.

\subsubsection{Incremental Learning}
The number of sample points is a key parameter since it determines a trade-off between computation time and the risk of overfitting. However, it seems that the network does not need many points at the beginning, but progressively requires more and more points during training. Figure~\ref{Fig:IncrementalLearning} illustrates the evolution of the JSR approximation of system~\eqref{Eq:Benchmark_dim_n} provided by a 2-hidden-layer neural network with $15$ neurons, where we now add successively new sample points to the training set. In this example, after combining the incremental learning with the regularization, the resulting network did not reach a wrong over-approximation value.

 \begin{figure}[t!]
    \centering
    \begin{subfigure}[t]{0.32\linewidth}
        \centering
        \includegraphics[scale=0.32]{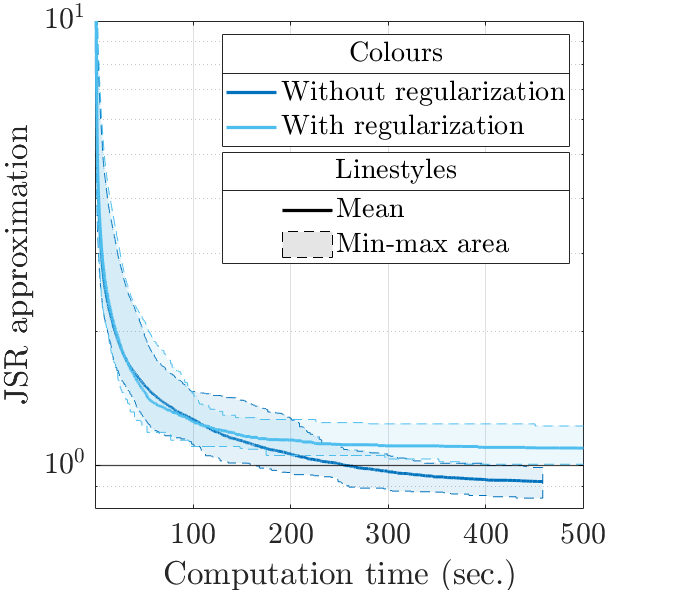}
        \caption{Comparison of the evolution with the computation time of the JSR approximation provided by a ReLU network with $1$ hidden layer and $15$ neurons with and without L$1$-\emph{regularization}.} 
        \label{Fig:Regularization}
    \end{subfigure}
    \hfill 
    \begin{subfigure}[t]{0.32\linewidth}
        \centering
        \includegraphics[scale=0.32]{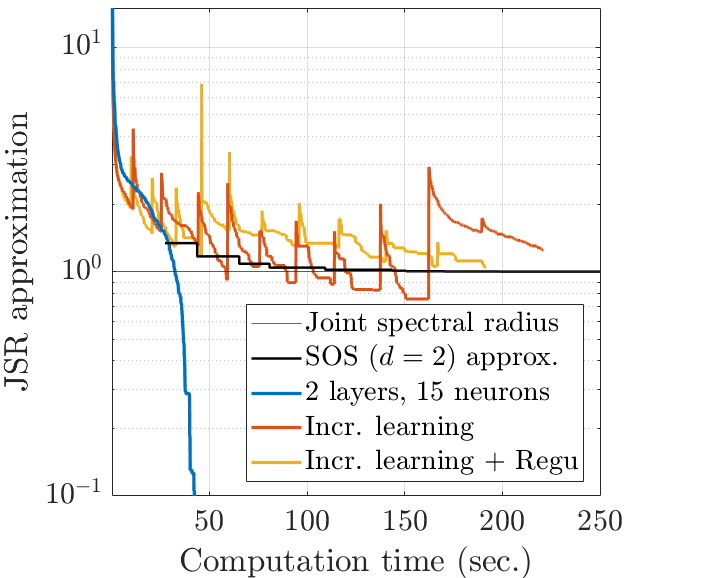}
        \caption{Comparison of the evolution with the computation time of the JSR approximation provided by a ReLU network with $1$ hidden layer and $15$ neurons with and without \emph{incremental learning}.} 
        \label{Fig:IncrementalLearning}
    \end{subfigure}
    \hfill
    \begin{subfigure}[t]{0.32\linewidth}
        \centering
        \includegraphics[scale=0.32]{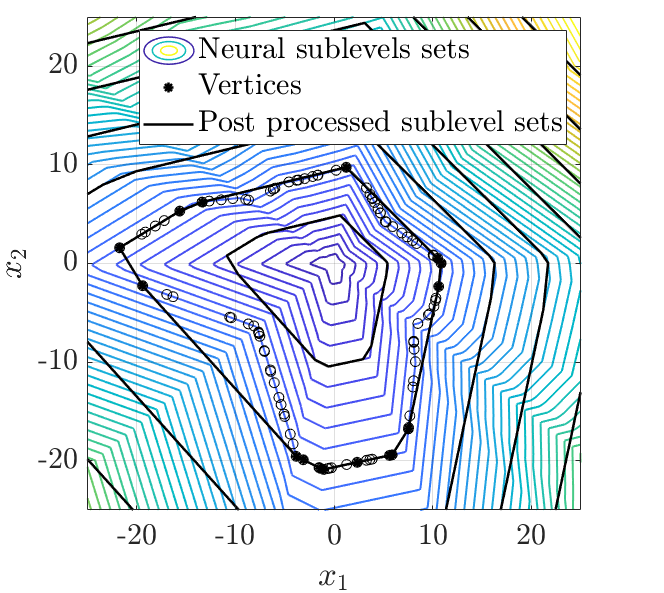}
        \caption{Comparison of the sublevel sets of the output of a $2$-hidden-layer neural network with $10$ neurons ($\rho_{NN(2,10)}(\Sigma_2) = 8.5788$) trained with $100$ sample points and the \emph{post processed} norm $N(\cdot)$ ($\rho_N(\Sigma_2) = 9.4157$).} 
        \label{Fig:PostProcess}
    \end{subfigure}
    \caption{Techniques to avoid overfitting in higher dimensions.}
\end{figure}

\subsubsection{Post Processing}
Until now we have seen with Examples~\ref{Exmp:Dimension2} and~\ref{Exmp:Dimension8} that neural networks can approximate well the JSR, but if badly tuned, they may under-estimate it. In safety-critical applications, we want a firm guarantee of stability: for this, we can leverage the information learnt by the network on the sample set $\mathcal{S}$ into an actual norm $\parallel \cdot \parallel_N$, which provides a valid upper bound, at the price of a more conservative estimate of the JSR. One way to do that is to consider the absolutely homogeneous function whose unit ball is defined by
\[ \overline{\mathcal{B}_{\parallel \cdot \parallel_N}} ~:= ~ conv \left\{ \frac{x_i}{V(x_i)} : x_i \in \mathcal{S} \right\},\]
such that $V$ and the induced norm $\parallel \cdot \parallel_N$ coincide on the sample set. See Figure~\ref{Fig:PostProcess} for an example with system~\eqref{Eq:ExmpNumSS}.

\section{Conclusions}

Motivated by recent developments in neural Lyapunov techniques, we have introduced for it a benchmark application and a theoretical framework, for the particular case of switched systems. These are a family of popular models for complex Cyber-Physical systems, that is algorithmically challenging, but rather well understood from a theoretical standpoint. 

We have shown that one can determine theoretical bounds for the accuracy of neural Lyapunov functions, as a function of the parameters of the network. These guarantees are competitive with classical SDP-based Lyapunov approaches in terms of number of variables. From the empirical point of view, we have shown that in practice as well, the approach is competitive while our neural networks were trained on simple personal computers, which leaves an important room for improvement. We have emphasized the problem of overfitting, and proposed several avenues for mitigating it. 

We see much possible further work, among which confirming the experiments at a more general level (different network architectures, different benchmark examples), potentially improving the learning algorithm (in particular, the loss function), and pushing further the approaches for mitigating overfitting.

\section{Acknowledgments} 
RJ is a FNRS honorary Research Associate. This project has received funding from the European Research Council (ERC) under the European Union's Horizon 2020 research and innovation programme under grant agreement No 864017 - L2C. Alec is supported by the EPSRC Centre for Doctoral Training in Autonomous Intelligent Machines and Systems (EP/S024050/1).

\bibliographystyle{plain}
\bibliography{aaai24}

\begin{thebibliography}{10}

\bibitem{Abate2021}
Alessandro Abate, Daniele Ahmed, Alec Edwards, Mirco Giacobbe, and Andrea
  Peruffo.
\newblock Fossil: A software tool for the formal synthesis of lyapunov
  functions and barrier certificates using neural networks.
\newblock In {\em Proceedings of the 24th International Conference on Hybrid
  Systems: Computation and Control}, HSCC '21, New York, NY, USA, 2021.
  Association for Computing Machinery.

\bibitem{Ahmadi2014}
Amir~Ali Ahmadi, Rapha\"{e}l~M. Jungers, Pablo~A. Parrilo, and Mardavij
  Roozbehani.
\newblock Joint spectral radius and path-complete graph lyapunov functions.
\newblock {\em SIAM Journal on Control and Optimization}, 52(1):687--717, 2014.

\bibitem{Alur2011}
Rajeev Alur, Alessandro D'Innocenzo, Karl~H. Johansson, George~J. Pappas, and
  Gera Weiss.
\newblock Compositional modeling and analysis of multi-hop control networks.
\newblock {\em IEEE Transactions on Automatic Control}, 56(10):2345--2357,
  2011.

\bibitem{Aroraetal-2018}
Raman Arora, Amitabh Basu, Poorya Mianjy, and Anirbit Mukherjee.
\newblock Understanding deep neural networks with rectified linear units.
\newblock In {\em International Conference on Learning Representations}, 2018.

\bibitem{AthaJung-2019}
N.~Athanasopoulos and R.~M. Jungers.
\newblock Polyhedral path-complete lyapunov functions.
\newblock In {\em 2019 IEEE 58th Conference on Decision and Control (CDC)},
  pages 3399--3404, 2019.

\bibitem{Barvinok-2012}
Alexander Barvinok.
\newblock {Thrifty Approximations of Convex Bodies by Polytopes}.
\newblock {\em International Mathematics Research Notices},
  2014(16):4341--4356, 04 2013.

\bibitem{berkenkamp2017safe}
Felix Berkenkamp, Matteo Turchetta, Angela~P. Schoellig, and Andreas Krause.
\newblock Safe model-based reinforcement learning with stability guarantees,
  2017.

\bibitem{Blondel2008}
Vincent Blondel and V.~Canterini.
\newblock Undecidable problems for probabilistic automata of fixed dimension.
\newblock {\em Theory of Computing Systems}, 36:231--245, 03 2008.

\bibitem{BlonNesThe-2005}
Vincent Blondel, Yurii Nesterov, and Jacques Theys.
\newblock On the accuracy of the ellipsoid norm approximation of the joint
  spectral radius.
\newblock {\em Linear Algebra and its Applications}, 394:91--107, 01 2005.

\bibitem{Blondel_boundedness_2000}
Vincent Blondel and John Tsitsiklis.
\newblock The boundedness of all products of a pair of matrices is undecidable.
\newblock {\em Systems \& Control Letters}, 41:135--140, 10 2000.

\bibitem{BLONDEL1997}
Vincent~D. Blondel and John~N. Tsitsiklis.
\newblock When is a pair of matrices mortal?
\newblock {\em Information Processing Letters}, 63(5):283--286, 1997.

\bibitem{BLONDEL2000}
Vincent~D. Blondel and John~N. Tsitsiklis.
\newblock A survey of computational complexity results in systems and control.
\newblock {\em Automatica}, 36(9):1249--1274, 2000.

\bibitem{Chang2020}
Ya-Chien Chang, Nima Roohi, and Sicun Gao.
\newblock Neural {{Lyapunov Control}}.
\newblock In {\em {{arXiv}}:2005.00611 [Cs, Eess, Stat]}, December 2020.

\bibitem{Dawson2021}
Charles Dawson, Zengyi Qin, Sicun Gao, and Chuchu Fan.
\newblock Safe nonlinear control using robust neural lyapunov-barrier
  functions.
\newblock In {\em Conference on Robot Learning}, 2021.

\bibitem{DebDelRosJun-2023}
Virginie Debauche, Matteo Della~Rossa, and Rapha\"{e}l Jungers.
\newblock Characterization of the ordering of path-complete stability
  certificates with addition-closed templates.
\newblock In {\em Proceedings of the 26th ACM International Conference on
  Hybrid Systems: Computation and Control}, HSCC '23, New York, NY, USA, 2023.
  Association for Computing Machinery.

\bibitem{Dereich2021}
Steffen Dereich and Sebastian Kassing.
\newblock On minimal representations of shallow relu networks.
\newblock {\em Neural networks : the official journal of the International
  Neural Network Society}, 148:121--128, 2021.

\bibitem{Farsi2022}
Milad Farsi, Yinan Li, Ye~Yuan, and Jun Liu.
\newblock A piecewise learning framework for control of unknown nonlinear
  systems with stability guarantees.
\newblock In {\em Learning for Dynamics and Control Conference, {L4DC} 2022,
  23-24 June 2022, Stanford University, Stanford, CA, {USA}}, volume 168 of
  {\em Proceedings of Machine Learning Research}, pages 830--843. {PMLR}, 2022.

\bibitem{fazlyab2021introduction}
Mahyar Fazlyab, Manfred Morari, and George~J Pappas.
\newblock An introduction to neural network analysis via semidefinite
  programming.
\newblock In {\em 2021 60th IEEE Conference on Decision and Control (CDC)},
  pages 6341--6350. IEEE, 2021.

\bibitem{HeretAl-2021}
Christoph Hertrich, Amitabh Basu, Marco Di~Summa, and Martin Skutella.
\newblock Towards lower bounds on the depth of relu neural networks.
\newblock In {\em Advances in Neural Information Processing Systems},
  volume~34, pages 3336--3348. Curran Associates, Inc., 2021.

\bibitem{HORNIK1991}
Kurt Hornik.
\newblock Approximation capabilities of multilayer feedforward networks.
\newblock {\em Neural Networks}, 4(2):251--257, 1991.

\bibitem{Brogliato2004}
Yildirim Hurmuzlu, Frank Génot, and Bernard Brogliato.
\newblock Modeling, stability and control of biped robots a general framework.
\newblock {\em Automatica}, 40:1647--1664, 10 2004.

\bibitem{RJ-2009}
R.M. Jungers.
\newblock {\em The Joint Spectral Radius: Theory and Applications}, volume 385
  of {\em Lecture Notes in Control and Information Sciences}.
\newblock Springer-Verlag, 2009.

\bibitem{Kozyakin1990}
Victor Kozyakin.
\newblock Algebraic unsolvability of a problem on the absolute stability of
  desynchronized systems.
\newblock {\em Automation and Remote Control}, 51:754--759, 01 1990.

\bibitem{Lauer2008}
Fabien Lauer and Gérard Bloch.
\newblock Switched and piecewise nonlinear hybrid system identification.
\newblock 04 2008.

\bibitem{Lee2020}
Donghwan Lee and Niao He.
\newblock A unified switching system perspective and convergence analysis of
  q-learning algorithms.
\newblock In {\em Advances in Neural Information Processing Systems},
  volume~33, pages 15556--15567. Curran Associates, Inc., 2020.

\bibitem{mcmullen_1970}
P.~McMullen.
\newblock The maximum numbers of faces of a convex polytope.
\newblock {\em Mathematika}, 17(2):179–184, 1970.

\bibitem{PARRILO2008}
Pablo~A. Parrilo and Ali Jadbabaie.
\newblock Approximation of the joint spectral radius using sum of squares.
\newblock {\em Linear Algebra and its Applications}, 428(10):2385--2402, 2008.
\newblock Special Issue on the Joint Spectral Radius: Theory, Methods and
  Applications.

\bibitem{Rota1960}
Gian-Carlo Rota and W.~Gilbert Strang.
\newblock A note on the joint spectral radius.
\newblock In {\em Proceedings of the Netherlands Academy}, pages 379--381,
  1960.

\bibitem{Zhang2023}
Songyuan Zhang, Yumeng Xiu, Guannan Qu, and Chuchu Fan.
\newblock Compositional neural certificates for networked dynamical systems.
\newblock In {\em Learning for Dynamics and Control Conference, {L4DC} 2023,
  15-16 June 2023, Philadelphia, PA, {USA}}, volume 211 of {\em Proceedings of
  Machine Learning Research}, pages 272--285. {PMLR}, 2023.

\end{thebibliography}

\end{document}